 \documentclass[journal,10pt,draftclsnofoot,onecolumn]{IEEEtran}

\usepackage{amsmath,amssymb,amsfonts}
\usepackage{amsthm}
\usepackage{hyperref}
\usepackage[labelformat=simple]{subcaption}

\usepackage{mathrsfs}
\usepackage{epsfig} %
\usepackage{graphicx}
\usepackage{graphics}
\usepackage{cite}
\usepackage{psfrag}
\usepackage{epstopdf}
\usepackage{tikz}
\PassOptionsToPackage{usenames,dvipsnames,svgnames}{xcolor} 

\usetikzlibrary{automata, arrows,decorations.pathmorphing,backgrounds,positioning,fit}
\usetikzlibrary{shapes.geometric}
\usetikzlibrary{circuits.ee,circuits.ee.IEC}

\definecolor{mycolor_peach}{RGB}{251,111,66}
\definecolor{mycolor_lightblue}{RGB}{8,180,238}
\definecolor{mycolor_darkblue}{RGB}{1,17,181}
\definecolor{mycolor_teal}{RGB}{18,150,155}
\definecolor{mycolor_green}{RGB}{12,195,82}
\definecolor{mycolor_lightgreen}{RGB}{8,180,238}

\usepackage{tikz-3dplot}
\usepackage{pgfplots}
\usepackage{cleveref}
\newcommand{\bm}[1]{{\boldsymbol{#1}}}

\theoremstyle{definition}

\newtheorem{example}{Example}

\theoremstyle{plain}
\newtheorem{theorem}{Theorem}
\newtheorem{proposition}{Proposition}

\theoremstyle{remark}

\newtheorem{remark}{Remark}

\begin{document}

\title{Gradient and Passive Circuit Structure in a Class of Non-linear Dynamics on a Graph}

\author{Herbert~Mangesius, 
      Jean-Charles Delvenne, Sanjoy K. Mitter
\thanks{H. Mangesius is with the Department of Electrical and Computer Engineering, Technische Universit\"{a}t M\"{u}nchen, Arcisstrasse 21, D-80209 Munich, Germany,
J.-C. Delvenne is with the ICTEAM and CORE, Universit\'{e} catholique de Louvain, 4 Avenue Lema\^{i}tre, B-1348 Louvain-la-Neuve, Belgium, and S.K. Mitter is with the Laboratory for Information and Decision Systems, Department of Electrical Engineering and Computer Science, MIT, Cambridge, MA 02139, U.S.A.  }
}


\maketitle

\begin{abstract}
We consider a class of non-linear dynamics on a graph that contains and generalizes various models from network systems and control and study convergence to uniform agreement states using gradient methods. 
In particular, under the assumption of detailed balance, we provide a method to formulate the governing ODE system in gradient descent form of sum-separable energy functions, which thus represent a class of Lyapunov functions; this class coincides with Csisz\'{a}r's information divergences.
Our approach bases on a transformation of the original problem to a mass-preserving transport problem and it reflects a little-noticed general structure result for passive network synthesis obtained by B.D.O. Anderson and P.J. Moylan in 1975.
The proposed gradient formulation extends known gradient results in dynamical systems obtained recently by M. Erbar and J. Maas in the context of porous medium equations.
Furthermore, we exhibit a novel relationship between inhomogeneous Markov chains and passive non-linear circuits through gradient systems, and show that passivity of resistor elements is equivalent to strict convexity of sum-separable stored energy.
Eventually, we discuss our results at the intersection of Markov chains and network systems under sinusoidal coupling.
\end{abstract}


\IEEEpeerreviewmaketitle

\section{Motivation}
 Gradient methods provide an elegant way to physics motivated modeling \cite{vdSchaftJeltsema2014} \cite{Otto2001} and are closely linked to passivity theory and the circuit concept \cite{vdSchaft2000} \cite{Willems1972a}. They are a basic tool in studying and designing non-linear systems on a graph, e.g., in distributed optimization \cite{Nedic2015} or in multi-robot problems such as coverage or formation control, cf., e.g., \cite{Egerstedt2010}, \cite{BulloCortesMartinez2009}, and references therein.

Another pillar in network system studies is the classical consensus problem \cite{Murray2007}. An equivalence between the dynamics (trajectories) of Markov chains and consensus networks has been source of recent advances in consensus theory \cite{Bolouki2014}. 
For LTI symmetric consensus networks such an equivalence has been linked to the averaging dynamics of unit capacitor RC circuits in \cite{Egerstedt2010} chap. 3.
Within the mathematics community, a static relationship is usually considered between Markov chains and electric circuits (resistor networks) \cite{Doyle1984}. The static (algebraic) circuit equations due to Kirchhoff and Ohm in fact are known to serve as generic structure underlying various scientific and computational problems, see, e.g., \cite{Strang2010} chap. 2.

Gradient formulations of Markov chains using sum-separable energy functions have been of recent interest in dynamical and non-linear systems \cite{ErbarMaas2014}\cite{Mielke2011}\cite{Maas2011}\cite{Zhou2012}. Interestingly, sum-separability of energy is an axiom in interconnected dissipative systems \cite{Willems1972a} and has origins in the circuit concept.

In this paper we bring these various concepts together in novel ways, based on a gradient structure for a class of non-linear dynamics on a graph that covers a wide range of prominent network system problems.

\section{Problem description and related literature \label{sec:pdes}}
Let $\mathsf{G}=(N,B,w)$ be a weighted directed graph, where
$N=\{1,2,\ldots,n\}$ is the set of nodes, $B=\{1,2,\ldots,b\}\subseteq N\times N$ denotes the set of branches whose elements are ordered pairs $(j,i)$ denoting an edge from node $j$  to $i$, and $w:B \to \mathbb{R}_{>0}$ is a weighting function, such that $w((j,i))=:w_{ij}$, if $(j,i) \in B$, else $w_{ij}=0$. 
Associated to a graph is the Laplace matrix
$\textbf{L}$,
defined component-wise as $[\textbf{L}]_{ij}=-w_{ij}$, $[\textbf{L}]_{ii}=\sum_j w_{ij}$. 
For strongly connected graphs, denote the positive left-eigenvector associated to the unique zero eigenvalue of the Laplacian by $\bm{c}$, and define $\textbf{C}:=\mathsf{diag}\{c_1,c_2,\ldots,c_n\}$. 

An important generalization of the symmetry condition on Laplacians that $\textbf{L}=\textbf{L}^\top$ is the particular type-symmetry that for some $\textbf{C}$, and $i,j\in N$,
\begin{equation}\label{eq:irredevp}
 c_i w_{ij}=c_jw_{ji} \Leftrightarrow \textbf{C}\textbf{L}=\textbf{L}^\top\textbf{C}.
\end{equation}
Equation \eqref{eq:irredevp} is known in the literature on Markov chains as detailed balance, or as reversibility  w.r.t. $\bm{c}$, cf., \cite{Norris1997} chap. 2.

\vspace*{0.3cm}

We consider the general class of dynamics on a graph $\mathsf{G}$ \mbox{described} component-wise by an ODE of the type
\begin{equation}\label{eq:nilinet}
\dot{x}_i=\sum_{j:(j,i)\in B}w_{ij}\,\phi(x_j,x_i), \ \  \ i \in N,
\end{equation}
where $\phi(\cdot,\cdot)$ is Lipschitz continuous,
$\phi(a,b)$ negative if $a<b$, zero iff $a=b$, positive if $a>b$, and $|\phi(a,b)|$ is increasing if $|a-b|$ is increasing. 

The class \eqref{eq:nilinet} includes many known network models:
The usual linear consensus system \cite{Murray2007} is obtained from setting
$\phi(x_j,x_i)=x_j-x_i$.
If $\phi(x_j,x_i)=f(x_j-x_i)(x_j-x_i)$, with $f(z)=f(-z)>0$,
then, the ODE \eqref{eq:nilinet} describes a continuous-time opinion dynamics \cite{Tilli2012}.
For instance, one may choose $f(z)=|\tanh(p\cdot z)|$, $p>0$, which is a good choice for \mbox{modeling} saturation phenomena in the interaction.
If $\phi(x_j,x_i)=\psi(x_j-x_i)$, $\psi(z)=-\psi(z)$, then we recover the non-linear consensus class introduced by Olfati-Saber and Murray in \cite{MurrayACC2003}, with $\psi=\sin$ a prominent instance.
Beyond the presented known \mbox{interaction} types, our model also includes couplings of the form $\phi(x_j,x_i)=g(x_j)-g(x_i)$,  where $g$ is an increasing function\footnote{or $\phi(x_j,x_i)=l(x_i)-l(x_j)$, where $l$ is a decreasing function}, e.g., $\ln (x),e^x, x^p$, $p>0$, on the respective domain of definition. The latter interaction type covers a discrete version of an equation system that models the non-linear diffusion of a gas in porous media, see \cite{ErbarMaas2014} (and \cite{Vazquez2007} for the continuous context).

\vspace*{0.3cm}

From an operational point of view, we are interested in bringing the ODE system \eqref{eq:nilinet}, under the assumption of detailed \mbox{balance}, into the gradient form
\begin{equation}\label{eq:gr}
\dot{\bm{q}}=-\textbf{K}(\bm{q})\nabla E(\bm{q}),
\end{equation}
where $\bm{q}$ is a suitable transform of the original state $\bm{x}$, $\textbf{K}(\cdot)$ is a symmetric, positive semi-definite matrix function that inherits the sparsity structure of the graph $\mathsf{G}$, and $E(\bm{q})$
is a sum-separable Lyapunov function.
    
 This structure defines gradient descent systems living on subspaces of $\mathbb{R}^n$, where
$-\nabla E(\cdot) \cdot\textbf{K}(\cdot)\nabla E(\cdot)$, describing locally the dissipation rate of $E$, is negative definite.
 In the context of \mbox{gradient} systems this structure is quite particular, as we impose the sparsity constraint given by $\mathsf{G}$,
 require sum-separability of the potential $E$, and do not require positive definiteness of the inverse metric $\textbf{K}$; these constraints are not usual from a \mbox{classical} gradient system point of view, cf., e.g., \cite{SimpsonPorco2014}, but turn out to be elementary in a passive circuits context.
\vspace*{0.3cm}

For particular cases of graph weightings and functions $\phi$, \mbox{gradient} structures for the class \eqref{eq:nilinet} have been established: For symmetric consensus systems, (i.e. $\textbf{L}=\textbf{L}^\top$, $\phi(x_j,x_i)=x_j-x_i$), it is well known that the network dynamics are a gradient descent of the (non-sum-separable) interaction potential $\frac{1}{2}\bm{x}^\top\textbf{L}\bm{x}$, see, e.g., \cite{Murray2007}. In \cite{vdSchaft2011} a port-Hamiltonian view as gradient descent of the sum-of-squares energy $ \frac{1}{2}\sum_{i\in N}x_i^2$ is presented. \mbox{Under} the less restrictive assumption of detailed balance weightings, the linear system dynamics (understood as Markov chain) has been formulated as gradient descent of free energy, resp. of relative entropy, in the works \cite{Zhou2012} and \cite{Maas2011}.
In \cite{ErbarMaas2014}, for systems with detailed balance weighting, and $\phi(x_j,x_i)=g(x_j)-g(x_i)$, $g$ increasing, a smooth gradient descent structure is presented for sum-separable energies $\sum_{i\in N}c_iH(x_i)$, $H$ being strictly convex and smooth on $\mathbb{R}_{>0}$.  
For the particular non-separable interaction case of having sinusoidal coupling, but symmetric weighting, gradient flow structures are represented, e.g., in \cite{Doerfler2014} or \cite{Sarlette2013}, where energy functions however are non-separable.

In the following we solve the general gradient representation problem and motivate the proposed structure requiring sum-separable energy functions from a passivity and circuit systems viewpoint.

%
%

\section{Gradient representation \label{sec:repres}}
With the following result we provide a procedure to bring a dynamics \eqref{eq:nilinet} into the form \eqref{eq:gr}. By that we characterize a \mbox{family} of sum-separable Lyapunov functions characterizing asymptotic stability of agreement states, i.e., states where all components are equal.

\begin{theorem}\label{thm:1}
Consider a network system dynamics governed by the protocol \eqref{eq:nilinet} on a strongly connected graph $\mathsf{G}$ such that detailed balance \eqref{eq:irredevp} holds for some $\textbf{C}$.
 Define the new state $\bm{q}:=\textbf{C}\bm{x}$, and consider the sum-separable function 
\begin{equation}\label{eq:thmliap}
E(\bm{q}):=\sum_{i\in N} c_iH(c_i^{-1}q_i), 
\end{equation}
where $H:\mathbb{R}\to\mathbb{R}$ is any $\mathscr{C}^2$-function, and
set $h(z):=\frac{\mathrm{d} H(z)}{ \mathrm{d} z}$.
If $H$ is strictly convex, then
the system can be represented as
\begin{equation}\label{thm:Kode}
\dot{\bm{q}}=-\textbf{K}(\bm{q}) \nabla E(\bm{q}),
\end{equation}
where $\textbf{K}(\cdot)$ is defined as the irreducible and symmetric Laplace matrix having components
\begin{equation}
[\textbf{K}]_{ij}:=\left\{\begin{tabular}{ll}
$-c_iw_{ij}\frac{\phi(x_j,x_i)}{h(x_j)-h(x_i)}$ & if  $j\not = i$,\\
$-\sum_{k=1, k\not = i}^n [\textbf{K}]_{ik}$ & if $i=j$.
\end{tabular} \right.
\end{equation}
The function \eqref{eq:thmliap} is a Lyapunov function establishing \mbox{asymptotic} stability of the \mbox{equilibrium} point $x_{\infty}\bm{1}$, with equilibrium value the weighted arithmetic mean  
 $x_{\infty}=\frac{\sum_{i\in N} c_ix_i(0)}{\sum_{i\in N}c_i}$.
 \end{theorem}
 \begin{proof}
 First, we observe that by the chain rule with $c_i^{-1}q_i=x_i$,
 \begin{equation}
 \frac{\partial }{\partial q_i}E(\bm{q})=c_i \frac{\partial H(x_i)}{\partial x_i}\frac{\partial x_i}{\partial q_i}=c_ih(x_i)c_i^{-1}=h(x_i).
 \end{equation}


The network dynamics \eqref{eq:nilinet} can be written equivalently as
\begin{equation}
\frac{1}{c_i}c_i\dot{x}_i=\sum_{j:(j,i)\in B} w_{ij}\phi(x_j,x_i)\Leftrightarrow
\dot{q}_i=\sum_{j:(j,i)\in B} c_iw_{ij}\phi(x_j,x_i).
\end{equation}
 Expanding by $h(x_j)-h(x_i)$ yields
\begin{align}
\dot{q}_i&=\sum_{j:(j,i)\in B} c_iw_{ij}\frac{\phi(x_j,x_i)}{h(x_j)-h(x_i)}\left(h(x_j)-h(x_i)\right) \\
&=\sum_{j:(j,i)\in B} [\textbf{K}]_{ij}\left(\frac{\partial }{\partial q_j}E(\bm{q})-\frac{\partial }{\partial q_i}E(\bm{q})\right) \\
 \Leftrightarrow \dot{\bm{q}}&=-\textbf{K}(\bm{q})\nabla E(\bm{q}),
\end{align}
where we use the identity $\sum_{j:(j,i)\in B}[\textbf{K}]_{ij}=-[\textbf{K}]_{ii}$.

Next, we show that the matrix $\textbf{K}$ is a symmetric, irreducible Laplace matrix. 
As $H$ is strictly convex and of type $\mathscr{C}^2$, $h$, as derivative of $H$, is an increasing function, 
 i.e., for any two real numbers $a,b$, $h(a)< h(b)$, whenever $a < b$.
 We observe that  $\phi(a,b)= \mathsf{sgn}(a-b)d(a,b)$, where $d$ is a Lipschitz continuous
  distance in $\mathbb{R}$.
 Now, $\frac{\phi(x_j,x_i)}{h(x_j)-h(x_i)}$ is symmetric in both arguments and positive for non-identical arguments, as
 \begin{align}
 \frac{\mathsf{sgn}(x_j-x_i)\cdot d(x_j,x_i)}{h(x_j)-h(x_i)}=& \frac{\mathsf{sgn}(x_j-x_i)\cdot d(x_j,x_i)}{\mathsf{sgn}(x_j-x_i) |h(x_j)-h(x_i)|} \\
  =& \frac{d(x_i,x_j)}{|h(x_i)-h(x_j)|}>0, \label{eq:symmratio}
 \end{align}
 where we used the fact that a distance is symmetric in both arguments and positive.
 By hypothesis \eqref{eq:irredevp}, for all $i\not = j$, $c_iw_{ij}=c_jw_{ji}$, so that
 \begin{equation}
 -[\textbf{K}]_{ij}=c_iw_{ij}\frac{\phi(x_j,x_i)}{h(x_j)-h(x_i)}=
 c_jw_{ji} \frac{\phi(x_i,x_j)}{h(x_i)-h(x_j)}=-[\textbf{K}]_{ji},
 \end{equation}
 and hence, $\textbf{K}$ is symmetric. It is also irreducible, as the underlying graph $\mathsf{G}$ is assumed to be strongly connected.
  
  Further, 
  for
 $x_j\to x_i$ the components $-[\textbf{K}]_{ij}$
are well-defined in the sense that their value remains positive and finite:
  For two real numbers $b,a$, $b>a$,
\begin{align}
\lim_{b\to a} \frac{\phi(b,a)}{h(b)-h(a)}\stackrel{\eqref{eq:symmratio}}{=}
&\lim_{b\to a} \frac{d(b,a)}{|h(b)-h(a)|} \\
= &\lim_{b\to a} \frac{d(b,a)}{|b-a|}\frac{|b-a|}{|h(b)-h(a)|}>0, \label{eq:pos}
\end{align}  
as the distance is Lipschitz continuous and increasing the further one moves away from $a$, and $h$ is an increasing function, so that the second fraction is positive as well. 
Hence, $[\textbf{K}]_{ij}$-elements are finite and positively bounded away from zero
for all parametrizations.

Convergence to an agreement state $x_{\infty}\bm{1}$ that is asymptotically stable follows from LaSalle's invariance principle:
\mbox{Stationarity} implies $\bm{0}=\textbf{K}(\bm{q})\nabla E(\bm{q})$, which is the case if and only if $\nabla E(\bm{q})\in \mathsf{ker}(\textbf{K})=\mathsf{span}\{\bm{1}\}$, as $\textbf{K}$ is a symmetric, irreducible Laplace matrix. 
Then, as $\textbf{K}$ is positive semi-definite, $E$ will decrease its value along solutions, as
\begin{equation}
\dot{E}(\bm{q})=-\nabla E(\bm{q})\cdot \textbf{K}(\bm{q})\nabla E(\bm{q})=-||\nabla E(\bm{q})||^2_{\textbf{K}},
\end{equation}  
  until a minimum is reached when $\nabla E(\bm{q})=\bm{h}(\bm{x}) \in \mathsf{span}\{\bm{1}\}$. 
 The function $h$ is an increasing function, hence bijective, so that the stationarity condition at a point $\bar{\bm{x}}$ is equivalent to $\bar{\bm{x}}\in \mathsf{span}\{\bm{1}\}$. 
Given an initial condition from a set defined by $\sum_{i} q_i=const.$, the set \mbox{$\mathcal{I}:=\{\bm{q}\in \mathbb{R}^n, \sum_{i} q_i=const.: \dot{E}(\bm{q})=0\}$} is a singleton, as $\bm{1}^\top \dot{\bm{q}}=\bm{1}^\top\textbf{K}(\bm{q})\nabla E(\bm{q})=0$, and \mbox{$\mathsf{span}\{\bm{1}\} \cap \mathcal{I}$} is a point $x_{\infty}\bm{1}$.
This shows that $E$ is a Lyapunov function for the considered network dynamics establishing asymptotic stability of an agreement state $x_{\infty}\bm{1}$.
  The  agreement value results from
  $\sum_{i} q_i(t)=const.$, so that $\sum_{i} c_ix_i(0)=x_{\infty}\sum_i c_i$ and therefore  $x_{\infty}= \frac{\sum_{i} c_ix_i(0)}{\sum_i c_i}$ the weighted arithmetic mean. 
 \end{proof}

In the following we provide a passive circuit interpretation of this gradient result, from where an intuitive meaning of the $\bm{q}$-variable system follows. We then present for our gradient construction the relationship to a structure result in passive network synthesis and show an equivalence between strict convexity of $E$ and passivity of resistor elements.

 \section{Synthetic circuit structure and passivity\label{sec:hma}}

 \subsection{Circuit formulation}
In circuit theory the dynamical behavior of a system is seen as the result of the interaction of a finite number of interconnected  circuit elements.
Circuit elements are single-input-single-output systems among which lossless, dynamical ones that can store energy and possess memory, e.g. capacitors in electric circuits, and memoryless, non-dynamic ones that dissipate energy, e.g. resistors, play important roles. Sum-separability of stored energy in this context thus arises naturally, as it is the sum of energies stored in individual lossless circuit elements.

We consider the lossless, dynamical multi-input-multi-output system
\begin{equation}\label{eq:MIMOsigmaN}
\left\{\begin{tabular}{ll}
$\dot{\bm{q}}(t)=\bm{u}_N(t)$\\
$\bm{y}_N(t)=\bm{h}(\bm{x}(t))= \nabla E(\bm{q})$,
\end{tabular}\right.
\end{equation}
which is an input-output model for elastic systems  \cite{Willems1972a}. In the circuits context it describes a generalized capacitor bank, where the 
 vector $\bm{q}$ collects charges in $n$ capacitors, each having a capacitance $c_i$, so that $\bm{u}_N$ is a vector of input currents, $\bm{y}_N$ a vector of generalized output voltages, and $E$ is the energy stored in the generalized capacitors.

\begin{example}[Passive LTI capacitor and electric energy]\label{ex:capacitor}
Electric energy stored in $n$ LTI passive capacitors
is given by
\begin{equation}
E(\bm{q})=\frac{1}{2}\sum_{i=1}^nc_i  v_{C,i}^2=
\frac{1}{2}\sum_{i=1}^n c_i  (c_i^{-1}q_i)^2,
\end{equation}
 which is of the particular form proposed in Theorem \ref{thm:1}.
 The \mbox{gradient} of electric energy w.r.t. charge as state is $\nabla E(\bm{q})=\bm{v}_C$, the capacitor voltage vector serving as output. Capacitors are current controlled, i.e., $\bm{u}_N=\bm{i}_C$, $i_{C,i}$ the capacitor $i$'s input current, so that we obtain the well-known input-output \mbox{characteristic} $s \textbf{C}\bm{v}_C=\bm{i}_C$, with $s$ a Laplace variable, and the state space representation $\dot{\bm{q}}=\bm{i}_C$, $\bm{v}_C=\nabla E(\bm{q})$.
\end{example}

In a resistor network where $b$ resistor elements (representing the branches in a graph) are interconnected at $n$ nodes, the \mbox{governing} equations follow Kirchhoff's voltage / current law (KVL / KCL) and Ohm's law \cite{Strang2010} chap. 2: Let $\textbf{B}$ be the usual $n\times b$ branch to node incidence matrix, and $\textbf{D}_B$ the $b\times b$ diagonal matrix collecting the positive values of $b$ conductances (inverse of resistances). The vector $\bm{v}_B$ collects the $b$ resistor voltages,  $\bm{i}_B$  the corresponding $b$ resistor currents,  $\bm{i}_N$ is the vector of node currents and $\bm{v}_N$ the vector of node voltages. We then have
\begin{equation}
\text{KVL}: \ \ \bm{v}_B:=-\textbf{B}^\top\bm{v}_N, \ \ \ \text{Ohm}: \ \ \bm{i}_B:=\textbf{D}_B\bm{v}_B, \ \ \ \text{KCL}: \ \ \bm{i}_N=\textbf{B}\bm{i}_B.
\end{equation}

Setting $\textbf{D}_B=\mathrm{diag}\{r_1^{-1},\ldots,r_b^{-1} \}$, with positive resistances at an edge $e=(j,i)$ given by
\begin{equation}\label{eq:resistance}
r_e:=\frac{1}{c_iw_{ij}}\frac{h(x_i)-h(x_j)}{\phi(x_i,x_j)},
\end{equation}
the gradient system proposed in Theorem \eqref{thm:1} represents an $RC$-circuit with dynamics described by voltage variables at the network nodes as
\begin{equation}
\textbf{C}\dot{\bm{v}}_N = \dot{\bm{q}} =  -\textbf{K}\nabla E(\bm{q})=-\textbf{B} \textbf{D}_B\textbf{B}^\top\bm{v}_N,
\end{equation}
where we use the fact that a symmetric, irreducible Laplacian can be factorized as $\textbf{K}=\textbf{B} \textbf{D}_B\textbf{B}^\top$.
This synthetic structure of a non-linear network dynamics \eqref{eq:nilinet} as $RC$-circuit is depicted in Fig. \ref{fig:synthstruct}.
 
\begin{figure}[t]
\centering
\begin{tikzpicture}[scale=1, circuit ee IEC]
\draw[color=black, rounded corners, fill=gray!30,line width=1pt] (-2.05,0) rectangle (4.15,3);
\draw[color=black,thick, fill=blue!35] (-0.6,1.7) rectangle (2.6,2.8);
 \draw[color=black,thick,fill=red!35] (-1.25,0.2) rectangle (3.4,1.3);
\draw[color=black,thick,fill=white] (-0.25,2) rectangle (0.25,2.6);
\node at (0,2.3) {$\frac{1}{s}$};
\draw[color=black,thick,fill=white]   (1.3,2) rectangle (2.3,2.6);
\node at (1.8,2.3) {$\nabla E(\cdot)$};
\draw[fill]  (-1.8,2.3) circle (0.07);

\node at (-1.25,2.55) {$\bm{i}_N$};
\node at (-0.45,2.55) {$\dot{\bm{q}}$};
\node at (0.9,2.55) {$\bm{q}$};
\node at (3.3,2.55) {$\bm{v}_N$};
\node at (2.15,1.05) {$\bm{v}_B$};
\node at (0,1.05) {$\bm{i}_B$};

\draw[color=black,thick,fill=white] (2.5,+1.1) rectangle (3.1,0.5) ;
\node at (2.8,0.8) {$\textbf{B}^\top$};

\draw[color=black,thick,fill=white] (-0.4,1.1) rectangle (-1.05,0.5) ;
\node at (-0.7,0.85) {$\textbf{B}$};

\draw[color=black,thick,fill=white] (0.3,1.1) rectangle (1.75,0.5) ;
\node at (1.05,0.8) {$\textbf{D}_B(\bm{q};\bm{c})$};

\draw[-latex] (-1.8,2.3) -- (-0.25,2.3) ;
\draw[-latex] (0.25,2.3) -- (1.3,2.3) ;
\draw[-latex] (2.3,2.3) -- (3.9,2.3)-- (3.9,0.8)-- (3.1,0.8) ;
\draw[-latex] (2.5,0.8) -- (1.75,0.8) ;
\draw[-latex] (0.3,0.8) -- (-0.4,0.8) ;
\draw[-latex](-1.05,0.8)--(-1.8,0.8)--(-1.8,2.25) ;

\node at (-1.6,2.05) {\LARGE-};
\end{tikzpicture}
\caption{Output feedback representation of the gradient dynamics as circuit:  Resistor network in Kirchhoff-Ohm factorized form (red), capacitor bank (blue)}
\label{fig:synthstruct}
\end{figure}
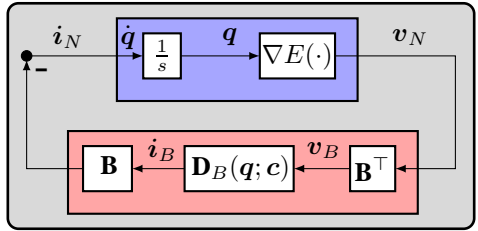


\begin{example}[LTI RC circuit and consensus dynamics]\label{ex:RC}
Consider a resistor network where at each node with one terminal an LTI capacitor of capacitance $c_k$ is connected, see Ex. \ref{ex:capacitor}.
Kirchhoff's and Ohm's laws lead to a local balance equation for currents flowing in and out at each node $k\in N$:
\begin{align}
i_{C,k}&=\sum_{e=(j,k)\in B}i_e= \sum_{e\in B} \frac{1}{r_{e}}v_e
=\sum_{j:(j,k)\in B} \frac{1}{r_{e}}(v_{C,j}-v_{C,k}) \\
\Leftrightarrow
 \dot{v}_k&=\sum_{j:(j,k)\in B} \frac{1}{c_k r_{kj}}(v_{C,j}-v_{C,k}),
\end{align}
where we used the capacitor equation $\dot{v}_{C,k}=c^{-1}_k i_{C,k}$.
Define the Laplacian $\textbf{L}$, with $[\textbf{L}]_{kj}=-\frac{1}{c_kr_{kj}}$,
and $[\textbf{L}]_{kk}=\sum_{j\not= k}\frac{1}{c_kr_{kj}}$. 
Then, the averaging $RC$-electric circuit dynamics of terminal voltages are a detailed-balance LTI consensus dynamics 
$\dot{\bm{v}}_C=-\textbf{L}\bm{v}_C$.
\end{example}

\begin{remark}
We note that the weights $w_{kj}=\frac{1}{c_kr_{kj}}$ in Ex. \ref{ex:RC} have a (physical) meaning in RC circuits: They represent inverse time-constants, or cut-off frequencies. 
\end{remark}

\subsection{Hill-Moylan-Anderson result, passivity and convexity}

D.C. Youla and P. Tissi showed in \cite{Youla1966} that a synthesized \mbox{dynamics} that solves the classical LTI network synthesis problem\footnote{Classical network synthesis is concerned with reproducing a prescribed LTI input-output behavior
in terms of a finite number of elementary passive, linear, ideal (lumped) circuit elements and a scheme for interconnecting them \cite{Anderson1975} \cite{Anderson2006}.} is structured as (negative) feedback system, in which
a dissipative and memoryless network controls 
a lossless and passive system that is comprised of decoupled unit capacitors and inductors\footnote{
The procedure to obtain this feedback representation is commonly referred to as reactance extraction, as the reactive, i.e., dynamical network elements are extracted from the composite system.}.

Anderson and P.J. Moylan in \cite{Anderson1975} and  Hill with Moylan in \cite{Hill1980}
proposed a structure result for non-linear ODE systems, which characterizes a realization of a vector field as lossless-memoryless decomposition similar to the well-known Youla-Tissi one for linear networks. For an equation system $\dot{\bm{x}}=\bm{f}(\bm{x})$ that generates solutions along which $E:\mathbb{R}^n\to\mathbb{R}$ is a differentiable Lyapunov function,
their non-linear analogue to extract a lossless and dynamic part leads to an output feedback structure as illustrated in Fig. \ref{fig:hmadec},
where $(\nabla E)^{-1}: \mathbb{R}^n\to\mathbb{R}^n$ denotes any (generally non-linear) function such that $[(\nabla E)^{-1} \circ \nabla E ](\bm{x})=\bm{x}$.

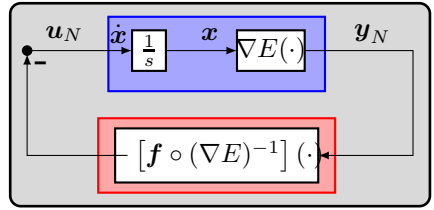
\begin{figure}[t]
\centering
\begin{tikzpicture}[scale=0.9, circuit ee IEC]
\draw[color=black, rounded corners, fill=gray!30,line width=1pt] (-2.05,0) rectangle (4.15,3);
\draw[color=blue,thick, fill=blue!35] (-0.6,1.7) rectangle (2.6,2.8);

 \draw[color=red,thick, fill=red!35] (-0.75,0.2) rectangle (2.75,1.3);
\draw[color=black,thick,fill=white] (-0.25,2) rectangle (0.25,2.6);
\node at (0,2.3) {$\frac{1}{s}$};
\draw[color=black,thick,fill=white]   (1.3,2) rectangle (2.3,2.6);
\node at (1.8,2.3) {$\nabla E(\cdot)$};
\draw[fill]  (-1.8,2.3) circle (0.07);

\node at (-1.25,2.55) {$\bm{u}_N$};
\node at (-0.45,2.55) {$\dot{\bm{x}}$};
\node at (0.9,2.55) {$\bm{x}$};
\node at (3.3,2.55) {$\bm{y}_N$};

 \draw[color=black,thick,fill=white] (2.5,1.15) rectangle (-0.5,0.35)  ;
\node at (1,0.75) {$-\left[ \bm{f}\circ(\nabla E)^{-1} \right] (\cdot)$};

 \draw[-latex] (-1.8,2.3)--(-0.25,2.3);
\draw[-latex] (2.3,2.3) -- (3.9,2.3)-- (3.9,0.75)--(2.5,0.75);
 \draw[-latex] (0.25,2.3) -- (1.3,2.3) ;
  \draw[-latex](-0.5,0.75)--(-1.8,0.75)--(-1.8,2.25) ;

 \node at (-1.6,2.05) {\LARGE-};

 \end{tikzpicture}
\caption{Hill-Moylan-Anderson decomposition for an ODE system $\dot{\bm{x}}=\bm{f}(\bm{x})$. }
\label{fig:hmadec} 
\end{figure}

The preceding circuit formulation of the gradient result in Theorem \ref{thm:1} yields an explicit implementation of this general structure result for the non-linear ODE system class \eqref{eq:nilinet}, i.e., for
$[\bm{f}(\bm{x})]_i=\sum_{j:(j,i)\in B}w_{ij}\,\phi(x_j,x_i),  i\in N$, where
 the inversion mechanism is realized by interconnected passive \mbox{resistor} \mbox{elements}:
The Laplacian structure resulting from Kirchhoff's laws allows to simply multiply pairwise interactions in the non-linear vector field with the reciprocal of energy gradient \mbox{differences}, to obtain the conductances 
\begin{equation}\label{eq:conductance}
r^{-1}_{ij}=c_iw_{ij}\frac{\phi(x_j,x_i)}{\frac{\partial}{\partial x_j} E(\bm{x})-\frac{\partial}{\partial x_i} E(\bm{x})},\ \ \ (j,i)\in B,
\end{equation}
with $w_{ij}$, $c_i$, $\phi$, $E$ as specified in section \ref{sec:pdes} and Theorem \ref{thm:1}.
In this reciprocal inversion mechanism, strict passivity of a resistor element is equivalent to positivity of resistance (resp. of conductance) and hence to dissipation of energy across a resistor edge connecting two non-identical potentials. In the context of the system class \eqref{eq:nilinet}, we observe an interplay between strict passivity of resistor elements and strict convexity of sum-separable energy:

\begin{proposition}\label{prop:passiveresistors}
Each resistor element characterized by a conductance \eqref{eq:conductance} is strictly passive if and only if the sum-separable energy $E$ as given in Theorem \ref{thm:1} is strictly convex.
\end{proposition}
\begin{proof}
Strict convexity of $E$ is equivalent to increasingness of the gradient component functions $h$.
If $h$ is increasing, then each $c_iw_{ij}\frac{\phi(x_j,x_i)}{h(x_j)-h(x_i)}=c_iw_{ij}\frac{d(x_j,x_i)}{|h(x_j)-h(x_i)|}>0$, see \eqref{eq:symmratio} with the preceding arguments. 
If $h$ were not (monotonously) increasing, but also decreasing at some part of the state space, then, in that region $\frac{h(x_j)-h(x_i)}{\phi(x_j,x_i)}<0$, i.e., the non-linear resistance would not be passive, but active. 
Conversely, if $h$ is constant on some interval on the real line, then
resistance would vanish, and so does dissipation and the dynamics, hence the resistance is not strictly passive (neither active).
\end{proof}

For sum-separable, strictly convex energy to be a Lyapunov function,  passivity of each (internal) resistor element is not \mbox{necessary} but
only sufficient. It is required that $\textbf{K}(\bm{q})$ is positive semi-definite,
so that 
\begin{equation}
\frac{\mathrm{d}}{\mathrm{d}t}E(\bm{q})=- \nabla E(\bm{q})\cdot \textbf{K}(\bm{q})\nabla E(\bm{q})=-||\nabla E(\bm{q})||_{\textbf{K}}^2<0 
\end{equation}
except at the equilibrium point. This dissipation inequality can be satisfied when $E$ is not strictly convex, i.e.,  individual non-linear resistors may locally (in some region on state space) be active (have negative resistance), as long as globally more energy is dissipated than created. 

\begin{remark}
In Hill, Moylan, Anderson's work it is assumed that the function $(\nabla E)^{-1}$ 
exists. Moylan, in \cite{Moylan2014} chap. 10, conjectures that convexity of $E$ is a sufficient condition for the \mbox{existence} of the function inverse. Strict convexity is in fact sufficient as a classical result from convex analysis shows: The function inverse $(\nabla E)^{-1}=\nabla E^\star$, where $E^\star$ is dual to $E$ in the sense of Young, see, e.g., \cite{Arnold1989}, chap. 3. 
\end{remark}

\subsection{Conservation of charge and Markov dynamics}
A pillar of the circuit concept is the conservation of total charge, see, e.g., \cite{Valkenburg1974} chap. 1. In the following we relate the non-linear charge dynamics associated to the class \eqref{eq:nilinet} to the dynamics of a (spatially) inhomogeneous Markov chain.

 With $\textbf{K}(\cdot)$ being a symmetric, irreducible Laplace matrix for all parametrizations, 
 $\mathsf{ker}(\textbf{K}(\cdot))=\mathsf{span}\{\bm{1}\}$, so that
 $\sum_{i\in N}q_i(t)=const.$ for all times $t\geq 0$, since $\bm{1}^\top\dot{\bm{q}}=0$. This recovers the conservation principle for charge in our circuit interpretation of Theorem \ref{thm:1}.
 Without loss of generality, we can choose $\bm{q}(0) \in \mathbb{R}^n_{> 0}$ such that $\sum_i q_i(0)=1$. The normalized $\bm{q}$-vector then also has the interpretation of a probability mass distribution on a discrete probability space: Each node $i\in N$ is a possible state, and $q_i(t)$ is the probability of a random walker on a graph $\mathsf{G}$ of being in state $i$ at time $t$, see, e.g., \cite{Lovasz1993}.
 
  The equation system describing the probability transport associated to the non-linear averaging dynamics \eqref{eq:nilinet} in $\bm{x}$-variables follows from the gradient formulation in Theorem \ref{thm:1} with the admissible choice  $H(x)=\frac{1}{2} x^2$, so that $\nabla E(\bm{q})=\bm{x}$:
  
\noindent   Define $\textbf{F}^\top(\cdot):=\textbf{K}(\cdot)\textbf{C}^{-1}$ and observe that $\textbf{F}=\textbf{C}^{-1}\textbf{K}$ is an irreducible (non-symmetric) Laplace matrix satisfying detailed balance. Then,
  \begin{equation}\label{eq:probdyn}
  \dot{\bm{q}}=-\textbf{K}(\bm{q}) \bm{x}=-\textbf{K}(\bm{q})\textbf{C}^{-1} \bm{q}=-\textbf{F}^\top(\bm{q})\bm{q}.
  \end{equation}
  where $\textbf{F}^\top$ is the infinitesimal generator of a (spatially inhomogeneous) Markov chain.
This Markov chain \mbox{asymptotically} reaches the invariant probability measure given by the \mbox{normalized} capacitances $\bm{c}$ with $\sum_{i\in N}c_i=1$.
  
 This equivalence between passive $RC$-circuits and Markov chains as dynamical systems bears the following novelties:
 
 i) The usual relation between electric circuits and Markov chains in the applied mathematics literature restricts to a static equivalence between a resistor network and the probability transition kernel of the Markov chain, cf., the seminal work \cite{Doyle1984} and references therein. In the engineering literature a dynamical relationship is known only for LTI symmetric consensus systems, which are equivalent to homogeneous, symmetric Markov chains and unit-capacitance RC-circuits, see, e.g., \cite{Egerstedt2010} chap. 3. We extend those results to dynamical, non-linear $RC$-circuits, where we show the relationship between detailed balance and non-unit capacitances.

ii)  With capacitances $\bm{c}$ normalized such that $\sum_i c_i=1$, stored energy $E(\bm{q})=\sum_i c_iH(c_i^{-1}q_i)$, $H$ strictly convex, corresponds to the class of information-divergences of a probability distribution $\bm{q}$ to the equilibrium distribution $\bm{c}$, introduced by Ali and Silvey \cite{AliSilvey1966} and Csisz\'{a}r, cf., \cite{Csiszar2004}. 
The usual technique to prove decreasingness of $E$ is based on Jensen's inequality \cite{Liese2006}, see also \cite{Willems1976}. Theorem \ref{thm:1} establishes this dissipation inequality for the class of Csisz\'{a}r's information divergences in a novel way, namely by exhibiting a passive RC circuit structure.


\begin{remark}
Other physical interpretations of the $\bm{q}$-variable (charge, probability) and $\bm{x}$-variable (voltage, density) system representations can be found in mass action chemical reaction networks, see, e.g., \cite{vdSchaft2013} and \cite{Mielke2011}.
\end{remark}

\section{Discussion \label{sec:disc}}
\subsection{Coupled oscillator models and electric power grids}
A generic model in the study of phase-coupled oscillator networks is given by the ODE system
on a graph $\mathsf{G}$,
\begin{equation}\label{eq:oscimodel}
\dot{\theta}_i=\omega_i+\sum_{j:(j,i)\in B}w_{ij}\sin (\theta_j-\theta_i), \ \ \ i\in N,
\end{equation}
where $\bm{\omega}=(\omega_1,\omega_2,\ldots,\omega_n) \in \mathbb{R}^n$ is the vector of
natural \mbox{(driving)} frequencies, and the state $\bm{\theta}\in \mathbb{T}^n$ is an $n$-vector of angles
as elements of the $n$-Torus.

If we set $w_{ij}=\frac{K}{n}$, $K>0$ for all $(j,i)\in B$, then \eqref{eq:oscimodel} represents 
Kuramoto's oscillator model \cite{Spigler2005}.
If $w_{ij}=\frac{|v_i||v_j| \Im(y_{ij})}{D_i}$, $|v_i|$ a voltage magnitude, $y_{ij}$ the complex admittance of a line $(j,i)$, and $D_i>0$ a damping parameter, then \eqref{eq:oscimodel} describes
a so-called droop control setup for frequency stabilization of generators in an electric power grid whose diffusively coupled voltage angles $\theta_i$ are driven by nominal power inputs $\omega_i$, see, e.g., \cite{Doerfler2014}.

Observe that the detailed balance condition  \eqref{eq:irredevp} naturally \mbox{applies} in this setting: In Ex. \ref{ex:RC} we saw that $c_iw_{ij}=c_jw_{ji}$ implies that weights $w_{ij}$ have the form
 $w_{ij}=\frac{1}{r_{ij} c_i}$, $r_{ij}=r_{ji}$. \mbox{Following} this RC circuits view, we may take $r_{ij}^{-1}=\Im(y_{ij})$ \mbox{having} unit siemens (one over ohm), and $c_i^{-1}=|v_i||v_j|/D_i$. Capacitance has unit farad $F=\frac{\text{ampere} \cdot \text{sec}}{\text{volt}}$, so that $D_i$ should carry the unit volt-ampere-seconds (an energy), which is \mbox{similar} to a measure of a power deviation per base frequency. This indeed matches the meaning of the factor $D_i$ in the droop control setting.

Using our reactance extraction approach we can write \eqref{eq:oscimodel} as (driven) gradient system, with $\bm{q}=\textbf{C}\bm{\theta}$, $\bm{\omega}_C:=\textbf{C}\bm{\omega}$, and the gradient $\bm{h}(\bm{\theta})=\nabla E(\bm{q})$, 
\begin{equation}\label{eq:oscidescent}
\dot{\bm{q}}=\bm{\omega}_C -\textbf{K}(\bm{\theta})\nabla E(\bm{q}),
\ \ \ [\textbf{K}]_{ij}=c_jw_{ij}\frac{\sin(\theta_j-\theta_i)}{h(\theta_j)-h(\theta_i)},
\end{equation}

For instance, if we choose $E=\sum_{i}c_i \theta_i^2=\sum_i\frac{1}{c_i}q_i^2$, (electric energy),
then, $\frac{\sin(\theta_j-\theta_i)}{\theta_j-\theta_i}=\mathsf{sinc}(\theta_j-\theta_i)$,
the sine cardinalis, or sampling function. We denote the corresponding inverse metric by $\textbf{K}_\mathsf{sinc}$.
Let $\bm{\omega}_C=c\bm{1}$, $c\in \mathbb{R}$, so that without loss of generality we can study the dynamics in a rotating frame at speed $c$ and set $\bm{\omega}=\bm{0}$ in \eqref{eq:oscidescent} \cite{Doerfler2014} \cite{Moreau2005} . 
Phase synchronization takes place if $\max_{i,j\in N} |\theta_j-\theta_i|<\pi$, 
because in that case $\mathsf{sinc}(\cdot)>0$, i.e., all non-linear resistances $[\textbf{K}_\mathsf{sinc}]^{-1}_{ij}$ are passive, 
so that $\lim_{t\to\infty} \bm{\theta}(t) \to \theta_{\infty}\bm{1}$, according to Theorem \ref{thm:1}. By that we recover a known phase synchronization result, see, e.g., \cite{Doerfler2014}, but via a passive circuits approach.


\vspace*{0.2cm}
As shown in section \ref{sec:repres}, we can bring the coupled oscillator model \eqref{eq:oscimodel} (in uniform rotating frame) into the form of a Markov chain dynamics.
The fact that the class of information divergences are Lyapunov functions, allows to explore mixing time bounds in Markov chains for convergence bounds in dynamics on a graph, which often are tighter than usual bounds based on the second largest eigenvalue of a Laplace matrix $\textbf{K}$, see \cite{Tetali2005} for an overview.

\subsection{Discrete De Bruijn's identity \label{ssec:debruijn}}
Two elementary quantities in information theory are differential entropy, a measure of the descriptive complexity of a random variable, and Fisher information, a measure of the \mbox{minimum} error in estimating a parameter from a distribution. Let $S \subseteq \mathbb{R}$ be the support set of a random variable $X$ of finite variance,
and let $x(\xi)>0$, $\forall \xi\in S$, be a probability density distribution\footnote{i.e., the gradient of the cumulative probability distribution function on $S$} 
for $X$. The differential entropy then is
defined as \cite{Cover1991} chap. 9,
\begin{equation}\label{eq:diffent}
\mathrm{Ent}(X):=-\int_S x(\xi)\ln x(\xi)\mathrm{d}\xi.
\end{equation}
The Fisher information of the distribution of $X$ can be defined as in \cite{Cover1991} chap. 16.6,
\begin{equation}\label{eq:fi}
\mathcal{J}(X):=\int_Sx(\xi)\left(\frac{\nabla x(\xi)}{x(\xi)}\right)^2\mathrm{d}\xi=
\int_S|\nabla \ln x(\xi)|^2x(\xi)\mathrm{d}\xi.
\end{equation}
This is a special form of the Fisher information, taken with respect to a translation parameter on the continuous support $S$, which does not involve
an explicit parameter as in its most \mbox{general} definition \cite{Cover1991} chap. 12, see also \cite{Verdu2006}.

De Bruijn's identity provides a relationship between these two quantities: For a
process $Y=X+\sqrt{t}Z$, where $Z$ is a normally distributed random variable, 
\begin{equation}\label{eq:debruijncontin}
\frac{\partial}{\partial t}\mathrm{Ent}(Y)=\frac{1}{2}\mathcal{J}(Y),
\end{equation}
see \cite{Cover1991} Theorem 16.6.2.

In the context of porous medium equations, Erbar and Maas in \cite{ErbarMaas2014} propose the discrete version of Fisher information
\begin{equation}
J(\bm{x}):=\frac{1}{2}\sum_{i,j\in N} c_iw_{ij}\phi(x_j,x_i)\left(h(x_j)-h(x_i)\right) ,
\end{equation}
where $\phi(x_j,x_i)=g(x_j)-g(x_i)$, $g$ an increasing function.
Note that the gradient of a function on a discrete space $(N,B,w)$ is given by the (weighted) difference of the function values at connected nodes.

For instance, if we choose  $g(x_j)-g(x_i)=x_j-x_i$, and relative entropy $E(\bm{q})=\sum_ic_iq_i\ln \frac{q_i}{c_i}$ as energy, so that $h=\ln$, we get
\begin{equation}\label{eq:discretedebruijn}
\frac{\mathrm{d}}{\mathrm{d}t} E(\bm{q})=-J(\bm{x})=\frac{1}{2}\sum_{i,j\in N} c_iw_{ij}\mathsf{lgm}(x_j,x_i)\left\vert \ln x_j-\ln x_i \right\vert^2,
\end{equation}
where  $\mathsf{lgm}(x_j,x_i):=\frac{x_j-x_i}{\ln x_j-\ln x_i}$ is the logarithmic mean of two positive variables, and $|\ln x_j-\ln x_i|^2$ is the discrete equivalent to $|\nabla \ln x|^2$ in \eqref{eq:fi}.

In the definition of Fisher information this gradient is integrated w.r.t.  $x \mathrm{d}\xi$.
On a discrete space gradients live on the set of edges $B$, while the density vector $\bm{x}$ is defined for elements indexed in the set $N$.
The logarithmic mean accounts for this lack of support in the discrete case: By the mean value theorem,
there exists a value $x_{ij} \in [x_i,x_j]$, (where we suppose that density components satisfy $x_i<x_j$), such that
\begin{equation}
\nabla \ln x_{ij} = \frac{\ln x_j-\ln x_i}{x_j-x_i} \ \ \Leftrightarrow \ \ x_{ij}=\mathsf{lgm}(x_j,x_i).
\end{equation}
Further,  $\mathsf{lgm}^{-1}(x_i,x_j)\equiv \int_0^1\frac{\mathrm{d}\xi}{x_i \xi+(1-\xi) x_j}$ \cite{Carlson1972}, so that an ``edge density'' $x_{ij}$ 
 can be seen as a (convex) interpolation of the density across edges $e=(j,i)\in B$ based on knowledge of density components $x_i,x_j$ defined on nodes $i,j\in N$.

Let us consider the discrete De Bruijn inequality for a system \ref{eq:nilinet} with coupling  $\phi(x_j,x_i=\sin(x_j-x_i)$. Then,
\begin{equation}
J(\bm{x})=\frac{1}{2}\sum_{i,j\in N} c_iw_{ij}\frac{\sin (x_j-x_i)}{\ln x_j-\ln x_i}\left\vert \ln x_j-\ln x_i \right\vert^2.
\end{equation}
Using a discrete chain rule we can write the corresponding edge density
$x_{ij}=\frac{\sin(x_j-x_i)}{x_j-x_i}\frac{x_j-x_i}{\ln x_j-\ln x_i}\triangleq \mathsf{sinc}(x_j-x_i)\mathsf{lgm}(x_j,x_i)$. This is the logarithmic mean modulated by a sampling function kernel that takes values between zero and one (on sets where each edge corresponds to a strictly passive resistor). 

\vspace*{0.2cm}
As far as we known, this connection between De Bruijn's identity in information theory and the dissipation equality \eqref{eq:discretedebruijn} as discrete version of it is novel.  
A discrete version on domains characterized by graphs is natural in applications, where high-dimensional data actually resides on nodes of graphs, and (non-linear) weightings may characterize an application's peculiarity in terms of the irregularity of the domain.
It would be interesting to further understand the role of discrete instead of discretized continuous information inequalities, as the presented one of De Bruijn, within the emerging field of signal processing on graphs \cite{Shuman2013}.

\section{Conclusion}
In this paper we established for a general class of non-linear dynamics on a graph with detailed balance weighting a family of gradient structures associated to sum-separable strictly convex energy functions. 
This structure extends known gradient results in dynamical systems to pairwise couplings involving non-separable non-linearity.
Based on our gradient formalism we made several connections between previously separated results in dynamical network systems, Markov chains, information and circuit theory, where at the heart lies a little-noticed structure result for non-linear network synthesis, due to B.D.O. Anderson, P.J. Moylan and D. Hill.

\section*{Acknowledgment}
The second author was supported by the Interuniversity Attraction Pole ``Dynamical Systems, Control and Optimization (DYSCO)'', initiated
by the Belgian State, Prime Minister's Office, and the Action de Recherche Concert\'ee funded by the Federation Wallonia-Brussels.
The third author was supported by the National Science Foundation under the grant EECS-1135843, ``CPS: Medium: Collaborative Research:'Smart Power Systems of the Future: Foundations for Understanding Volatility and Improving Operational \mbox{Reliability}'''

\bibliographystyle{ieeetr}
{\small
\bibliography{IntEnergyVarConsBib}
}

%



\end{document}